\documentclass[12pt]{amsart}
\usepackage{amscd}
\usepackage{amssymb}
\usepackage{a4wide}
\usepackage{amstext}
\usepackage{amsthm}
\usepackage{mathrsfs}

\usepackage{tikz}
\DeclareGraphicsExtensions{.pdf,.png,.jpg}
\usetikzlibrary{graphs}
\usetikzlibrary{arrows,positioning,calc}

\usepackage[final]{hyperref}
\hypersetup{unicode= false, colorlinks=true, linkcolor=blue,
anchorcolor=blue, citecolor=green, filecolor=red, menucolor=blue,
pagecolor=blue, urlcolor=blue} 
\numberwithin{equation}{section}

\newcommand{\abs}[1]{| #1 |}

\renewcommand{\phi}{\varphi}

\newcommand{\co}{\mathbb{C}}

\newcommand{\T}{\mathbb{T}}
\newcommand{\D}{\mathbb{D}}

\newtheorem{Thm}{Theorem}[section]
\newtheorem{theorem}[Thm]{Theorem}
\newtheorem{lemma}[Thm]{Lemma}
\newtheorem{proposition}[Thm]{Proposition}

\newtheorem{remark}[Thm]{Remark}

\begin{document}
\sloppy

\title{Bernstein-type inequalities for mean $n$-valent functions}

\author{Anton~Baranov}
\address{Department of Mathematics and Mechanics, St. Petersburg State University, 28,
Universitetskii prosp., 198504 Staryi Petergof, Russia}
\email{anton.d.baranov@gmail.com} 
\author{Ilgiz~Kayumov}
\address{Department of Mathematics and Computer Science, St. Petersburg State University,
Russia, 14 line of the VO, 29B, 199178 St. Petersburg, Russia}
\address{Kazan Federal University, 420000, Kazan, Russia}
\email{ikayumov@gmail.com}
\author{Rachid Zarouf}
\address{Aix-Marseille University, Laboratoire ADEF, Campus Universitaire de Saint-
J\'erome, 52 Avenue Escadrille Normandie Niemen, 13013 Marseille, France}
\address{CPT, Aix-Marseille Universit\'e, Universit\'e de Toulon, Marseille, France}
\email{rachid.zarouf@univ-amu.fr}
\begin{abstract}
We derive new integral estimates  of the derivatives of mean $n$-valent functions in the unit disk. Our results develop and complement estimates obtained by  E.\,P.~Dolzhenko and A.\,A.~Pekarskii, as well as recent inequalities obtained by the authors. As an application, we improve
some inverse theorems of rational approximation due to Dolzhenko.
\end{abstract}

\thanks{The work of I.R. Kayumov in Sections 2, 3 and 6 is supported by the Russian Science
Foundation grant No. 23-11-00066. The work of A.D. Baranov is supported by 
Ministry of Science and Higher Education of the Russian Federation under agreement No. 075-15-2024-631.}

\maketitle

\section{Introduction}

Estimates of the norms of the derivatives of rational functions in various spaces of analytic functions
are among classical problems of function theory. Of special interest are estimates depending only 
on the degree of a rational function but not on the localization of its zeros. Many deep estimates were obtained,
e.g., by E.\,P.~Dolzhenko \cite{Dol}, V.\,V.~Peller \cite{pel}, A.\,A.~Pekarskii \cite{pek}, V.\,I.~Danchenko 
\cite{dan0, dan}, E.~Dynkin \cite{dyn1, dyn2} and many other authors. 

Since rational functions of degree $n$ are, in particular, $n$-valent functions, it is a natural problem
to extend known inequalities to this, more general, class of functions. The following inequalities were obtained
recently in \cite{BaKa1} (the symbol ${\rm d}\mathcal{A}$ denotes the standard area measure
and $A^p(G)$ stands for the usual Bergman space in $G$).

\begin{theorem}
\label{aaa}
Let $G$ be a bounded domain in $\mathbb{C}$  with a rectifiable boundary. Then for any $n$-valent bounded function
$f$ in $G$ one has 
\begin{equation}
\label{domp}
\|f'\|_{A^p(G)}^p = \int_{G} |f'(w)|^p\, {\rm d}\mathcal{A}(w) \leq C n^{p-1} \|f\|^p_{H^\infty(G)}, \qquad 1<p\le 2, 
\end{equation} 
\begin{equation}
\label{domp1}
\|f'\|_{A^1(G)} = \int_{G} |f'(w)|\, {\rm d}\mathcal{A}(w) \leq C \sqrt{\log (n+1)} \, \|f\|_{H^\infty(G)}.
\end{equation}
Here $\|f\|_{H^\infty(G)} = \sup_{w\in G} |f(w)|$ and the constants depend only on $G$ and $p$, but not on $f$.
\end{theorem}

Theorem \ref{aaa} generalizes inequalities proved by Dolzhenko \cite{Dol} for rational functions and under strong 
regularity assumptions (including $C^2$-smoothness) on the boundary of the domain. Moreover, in the case of $p=1$ a substantial improvement
for the growth order was obtained (in \cite{Dol} the estimate was proved with $\log n$ on the right). Inequalities 
\eqref{domp} and \eqref{domp1} are sharp already
for the case of the unit disc. For $p>1$ the trivial example is given by $f(z) = z^n$. The sharpness of the growth 
$\sqrt{\log n}$ is attained on polynomials or Blaschke products of degree $n$, but the examples are complicated and are based
on an ingenious construction of Bloch functions with prescribed growth due to R.~Ba\~nuelos and C.\,N.~Moore \cite{BanMoo} 
(see \cite{BaKa} for details).

In the case of rational functions in the unit disc more general results were obtained by Pekarskii \cite{pek} who showed 
that the $H^\infty$ norm on the right can be replaced by a certain $H^p$ norm. We formulate a part of 
\cite[Theorem 2.1]{pek}.

\begin{theorem}
Let $p>1$, $\alpha>0$ and $\sigma = \frac{p}{\alpha p+1}$. Then for any rational function $f$ of degree at most $n$ with poles
in $\{|z|>1\}$ one has
\begin{equation}
\label{lll}
\|f\|^\sigma_{B_\sigma^\alpha} = 
\int_\D |f'(z)|^\sigma (1-|z|^2)^{(1-\alpha)\sigma -1} {\rm d}\mathcal{A}(z) \le Cn^{\alpha \sigma }\|f\|^\sigma_{H^p}.
\end{equation}
Here the constant $C$ depends only on the parameters $p, \alpha$ and $H^p$ denotes the standard Hardy space in $\D$.
\end{theorem}

Note that on the left we have a Besov $B_\sigma^\alpha$ seminorm. Non-weighted Bergman norm of the derivative
corresponds to the case $\alpha = \frac{\sigma-1}{\sigma}$. In this case one has
$$
\int_\D |f'(z)|^\sigma  {\rm d}\mathcal{A}(z) \le Cn^{\sigma-1 }\|f\|^\sigma_{H^p}, \qquad p=\frac{\sigma}{2-\sigma}, \ 
\sigma <2.
$$

Thus, Pekarskii's theorem does not cover the case of the usual mean of the derivative $\int_\D |f'(z)|{\rm d}\mathcal{A}(z)$. 
The aim of this note is to obtain an estimate for this mean not only for rational functions, but also for mean $n$-valent functions. 
Recall that a function $f$ analytic in a domain $G$ is said to be $n$-valent in $G$ if for any $w\in\co$ the equation
$f(z) =w$ has at most $n$ solutions in $G$. Denote by $n(w)$ the number of solutions of this equation in $G$ (counting multiplicities). A function $f$ is said to be mean $n$-valent in $G$ if  for any $R>0$
$$
\frac{1}{\pi}\int_0^R\int_0^{2\pi} n(re^{it}) {\rm d}\phi \,r {\rm d}r \le n R^2.
$$
In this definition $n$ can be an arbitrary positive number, not necessarily an integer.
For the definition and theory of mean $n$-valent functions see \cite{hay}.

The following is our main result.

\begin{theorem}
\label{th1}
For any $p>1$ there exists a constant $C=C(p)>0$ such that, for any mean $n$-valent in $\D$ function $f\in H^p$, 
$$
\int_\D |f'(z)|{\rm d}\mathcal{A}(z) \le C \sqrt{\log(n+1)} \|f\|_{H^p}.
$$
\end{theorem}

In Section \ref{poj} we prove a slightly more general statement for mean $n$-valent functions 
in sufficiently regular domains.

The case $p=1$ is not covered by the methods of this note. As far as we know, the only known result for the case $p=1$
is an inequality due to Ch.~Pommerenke \cite{pom62} which gives a much faster growth. Namely,  it is shown
in \cite{pom62} that for any mean $n$-valent function $f$ such that $f\in H^p$, $p<2$, one has
$$
\int_0^1 \bigg( \int_0^{2\pi} |f'(re^{it})| {\rm d}\phi \bigg)^p {\rm d}r \le C(p) n^{p/2} \|f\|^p_{H^p}.
$$
For $p=1$ this gives the growth of order $\sqrt{n}$ in place of $\sqrt{\log n}$.
It is an interesting problem to find sharp order of the means of the derivatives for $f\in H^1$.

Another interesting question is to obtain an analog of Pekarskii's inequalities for $n$-valent or mean $n$-valent functions. 
Applying our methods to this question gives only partial results which we believe can be improved; however we include them as an illustration of the method. For sufficiently large $p$ the inequality differs from \eqref{lll} only by a logarithmic factor. 

\begin{theorem}
\label{th2}
Let $p>1$, $0<\alpha<1/2$ and $\sigma<2$. If $p\ge \frac{2\sigma}{2-\sigma}$, then,
for any mean $n$-valent in $\D$ function $f\in H^p$, 
\begin{equation}
\label{gott1}
\|f\|^\sigma_{B_\sigma^\alpha} \le C(p, \alpha, \sigma) n^{\alpha\sigma} (\log n)^{\sigma/2}\|f\|^\sigma_{H^p}.
\end{equation}
If $p< \frac{2\sigma}{2-\sigma}$ and $\sigma < \frac{p}{\alpha p+1}$, then
\begin{equation}
\label{gott2}
\|f\|^\sigma_{B_\sigma^\alpha} \le C(p, \alpha, \sigma)  n^{\alpha\sigma +\sigma/p +\sigma/2-1} \|f\|^\sigma_{H^p}.
\end{equation}
\end{theorem}

As a byproduct of our estimates we obtain the following result which can be of independent interest. As was shown in \cite{BaKa},
for Blaschke products of degree $n$ the quantity $\int_\D |B'(z)|{\rm d}\mathcal{A}(z)$ can grow as $c\sqrt{\log n}$ (for some absolute positive constant $c$). However, the estimate remains true if we multiply $B'$ by some function from a Hardy space.

\begin{theorem}
\label{th3} 
Let $B$ be a finite Blaschke product of degree $n$ and $p\in(1,\infty)$. Then there exists $C=C(p)$ such that,
for any $g\in H^p$,
\[
\int_\D |B'(z) g(z)|{\rm d}\mathcal{A}(z) \le C \sqrt{\log (n+1)}\|g\|_{H^{p}}.
\]
\end{theorem}

The case $p=1$ in Theorem \ref{th3} also remains open.
\medskip

Both Dolzhenko \cite{Dol} and Pekarskii \cite{pek} used estimates of integral means 
of the derivatives to obtain inverse theorems of rational approximation. 
For  a domain $G$ denote by $\mathcal{R}_{n}(G)$ the set of rational functions 
of degree at most $n$ with poles outside of the closure $\overline{G}$ of $G$. Given a
Banach space $X$ of holomorphic functions in $G$ (which contains $H^\infty(G)$) we consider
the best approximation of $f$ in $X$ by rational functions in $\mathcal{R}_{n}$:
\[
R_{n}(f,X)=\inf\left\{ \left\Vert f-g\right\Vert _{X}\::\:g\,\in\mathcal{R}_{n}(G)\right\}.
\]
Analogously, one can define best approximations by polynomials $E_n(f, X)$ or
best approximations by $n$-valent functions $V_n(f,X)$; in the latter case the supremum is taken over
all $n$-valent functions $g\in X$.

Among other results, Dolzhenko \cite[Theorem 3.6]{Dol} showed that for $1\le p\le 2$ 
and for sufficiently smooth domains
(in the case $p=1$ the domain $G$ is assumed to be bounded) the condition
\begin{equation} 
\label{tig}
\sum_{n=1}^\infty \frac{R_n(f, H^\infty(G))}{n^{1/p}} <\infty
\end{equation}
implies that $f'\in A^p(G)$, where $A^p(G)$ is the standard Bergman space in $G$. However, in the case $p=1$
this result again can be improved and an essentially sharp order of decay can be found
(see Theorem \ref{jkl} below).
Moreover, in the case of the disc $H^\infty$ can be replaced by any $H^p$ with $p>1$.

\begin{theorem}
\label{tinv}
Let $p\in(1,\infty]$ and $f\in H^{p}.$ If 
\begin{equation} 
\label{et}
\sum_{n=2}^\infty \frac{R_{n}(f,\,H^{p})}{n \sqrt{\log n}}<\infty,
\end{equation}
then $f'\in A^{1}(\D).$ 
\end{theorem}

Using a construction similar to the construction proposed by J.\,E.~Littlewood \cite{Lit} one can show
that the condition of Theorem \ref{tinv} is optimal in the following sense. 

\begin{proposition}
\label{tinw}
Let $\phi$ be a positive function on $[1, \infty)$ such that 
$\lim_{x\to \infty} \phi(x)=0$. Then there exists a function $f$ such that 
$f\in H^p$ for any $p<\infty$,
$$
\sum_{n=2}^\infty \frac{E_{n}(f,\,H^{p})}{n \sqrt{\log n}} \phi(n) <\infty,
$$
but $f' \notin A^1$.
\end{proposition}

Note also that in the formulation of Theorem \ref{tinv} $R_n(f, H^p)$ 
can be replaced by a smaller quantity $V_n(f, H^p)$.
\bigskip


\section{Estimates far from the boundary}


In what follows we will use the classical Littlewood--Paley theorem (see \cite{Pav}, p.332). 
The latter result asserts that given $p>0$ and $f\in H^p$, if 
$$
S(f)(e^{it}) = \bigg(\int_0^{1}(1-r)|f'(re^{it})|^{2} {\rm d}r\bigg)^{1/2},
$$
then 
$$
\|S(f) \|_{L^p(\T)} \le C(p) \|f\|_{H^p}.
$$


\begin{lemma}
\label{lem1}
Let $f\in H^1$, $g\in H^\infty$ or  $f, g\in H^2$ and $K>0$. Then
$$
\int_{\abs{z}<1-1/n^K} |f'(z) g(z)|{\rm d}\mathcal{A}(z) \le C(K ) \sqrt{\log (n+1)} 
\begin{cases}
\|f\|_{H^1}\|g\|_{H^\infty}, \\
\|f\|_{H^2} \|g\|_{H^2},
\end{cases}
$$
respectively.
\end{lemma}

\begin{proof}
Assume first that $f\in H^1$, $g\in H^\infty$. By the Cauchy--Schwarz inequality we get
$$
\begin{aligned}
\int_0^{2\pi}\int_0^{1-1/n^K}|f'(re^{it}) g(re^{it})| r\, {\rm d}r dt & \leq  \|g\|_{H^\infty}
\int_0^{2\pi} S(f)(e^{it}) \left(\int_0^{1-1/n^K}\frac{r^2 {\rm d}r}{1-r}\right)^{1/2} {\rm d}t  \\
& \le  C(K ) \sqrt{\log (n+1)}\|f\|_{H^1}\|g\|_{H^\infty}
\end{aligned}
$$
by the Littlewood--Paley theorem. Now let $f,g\in H^2$. Again by the Cauchy--Schwarz inequality
$$
\begin{aligned}
& \int_{\abs{z}<1-1/n^K} |f'(z) g(z)|{\rm d}\mathcal{A}(z) \\
&\le 
\bigg( \int_{\abs{z}<1-1/n^K} |f'(z)|^2(1-|z|^2) {\rm d}\mathcal{A}(z)\bigg)^{1/2} 
\bigg( \int_{\abs{z}<1-1/n^K} \frac{|g(z)|^2}{1-|z|^2}{\rm d}\mathcal{A}(z)\bigg)^{1/2}  \\
& \le  C(K ) \sqrt{\log (n+1)}\|f\|_{H^2}\|g\|_{H^2},
\end{aligned}
$$
where the last inequality follows from the estimates 
$$
\int_{\abs{z}<1-1/n^{K}}|f'(z)|^{2}(1-|z|^{2}){\rm d}\mathcal{A}(z) 
\leq\int_{\mathbb{D}}|f'(z)|^{2}(1-|z|^{2}){\rm d}\mathcal{A}(z) \leq\|f\|_{H^{2}}^{2},
$$
and
\[
\int_{\abs{z}<1-1/n^{K}}\frac{|g(z)|^{2}}{1-|z|^{2}}{\rm d}\mathcal{A}(z)\leq
C\log(n+1)\|g\|_{H^{2}}^{2},
\]
$C$ being a nonnegative constant that depends on $K$ only. 
\end{proof}

\begin{lemma}
\label{lem1-0}
Let $\sigma<2$, $\alpha >0$, $p\ge 2$ and $K>0$. Then there exists $C=C(K, \sigma, \alpha, p)$ such that, for $f\in H^p$, 
$$
J_K(f) = \int_{\abs{z}<1-1/n^K} 
|f'(z)|^\sigma (1-|z|^2)^{(1-\alpha)\sigma -1} {\rm d}\mathcal{A}(z) \le Cn^{K\alpha \sigma} \|f\|_{H^p}^\sigma.
$$
\end{lemma}

\begin{proof}
By the H\"older inequality with the exponents $\frac{2}{\sigma}$ and $\frac{2}{2-\sigma}$ we have
$$
J_K(f) \le
\int_0^{2\pi} \big(S(f)(t)\big)^\sigma \left(\int_0^{1-1/n^K}\frac{{\rm d}r}
{(1-r)^\beta} \right)^{(2-\sigma)/2} dt,
$$
where $\beta = (-\sigma/2 + 1+\alpha\sigma)\frac{2}{2-\sigma} = 1 + \frac{2\alpha\sigma }{2-\sigma}$.  
The conclusion follows by a direct computation.
\end{proof}
\bigskip


\section{Hayman's inequality and estimates near the boundary}
\label{poj}

To estimate the integral in a narrow annulus near the boundary we will use an estimate due to 
W.K.~Hayman \cite[Lemma 3.1]{hay}. For a function $f$ analytic in the disc, 
we put $M(r,f) = \max_{|z|=r} |f(z)|$. Then for any
mean $n$-valent function $f$, for any $\lambda\in (0,2)$ and any $r\in (1/2,1)$ there exists $\tilde r$ such that 
$2r-1 \le \tilde r \le r$ (that is, $(1-\tilde r)/2 \le 1-r \le 1-\tilde r$) with
\begin{equation}
\label{hain}
\int_0^{2\pi} |f'(\tilde re^{it})|^2 |f(\tilde re^{it})|^{\lambda-2} {\rm d}t\le 4n\frac{\big(M(r, f)\big)^{\lambda}}{\lambda(1-r)}.
\end{equation}
The following estimate is an immediate corollary of  \eqref{hain}.

\begin{lemma}
\label{lem2}
Let $1\le p<2$ and $f\in H^p$ be mean $n$-valent. Then
$$
\int_0^{2\pi} |f'(re^{it})| {\rm d}t\le C(p) \frac{n^{1/2}}{(1-r)^{1/p}}\|f\|_{H^p}.
$$
\end{lemma}

\begin{proof}
Note that for $f\in H^p$ one has $M(r,f) \le C(p) \frac{\|f\|_{H^p}}{(1-r)^{1/p}}$.
Let $1/2 < r<1$ and $r_0 =\frac{1+r}{2}$. Apply Hayman's estimate with $\lambda = 2-p$ to $r_0$ 
to get the corresponding
$\tilde r \ge 2r_0 -1 =r$. Then, using Cauchy-Schwarz inequality once again,
$$
\begin{aligned}
\int_0^{2\pi} |f'(\tilde re^{it})| {\rm d}t& \le
\bigg( \int_0^{2\pi} |f'(\tilde re^{it})|^2 |f(\tilde re^{it})|^{\lambda-2} {\rm d}t\bigg)^{1/2}
\bigg( \int_0^{2\pi} |f(\tilde re^{it})|^{2-\lambda} {\rm d}t\bigg)^{1/2} \\
& \le \bigg(4n\frac{\big(M(r_0, f)\big)^{\lambda}}{\lambda(1-r_0)} \bigg)^{1/2} \|f\|^{p/2}_{H^p} 
\le C(p) \frac{n^{1/2}}{(1-r)^{1/2+\lambda/(2p)}}\|f\|_{H^p}.
\end{aligned}
$$
Since  $1/2+\lambda/(2p) = 1/p$ and the integral $\int_0^{2\pi} |f'(re^{it})| {\rm d} t$ is an increasing function 
of $r$, the lemma is proved.
\end{proof}

\begin{proof}[Proof of Theorem \ref{th1}]
Let $f\in H^p$ with $1<p<2$. By Lemma \ref{lem1} applied to $g\equiv 1$, we have, for any $K>0$,
$$
\int_{\abs{z}<1-1/n^K} |f'(z)|{\rm d}\mathcal{A}(z) \le C(K ) \sqrt{\log (n+1)} \|f\|_{H^1}.
$$
By Lemma \ref{lem2},
$$
\int_{1-1/n^K \le \abs{z} <1} |f'(z)|{\rm d}\mathcal{A}(z)
 \le C(p) \|f\|_{H^p} n^{1/2} \int_{1-1/n^K}^1 \frac{{\rm d}r}{(1-r)^{1/p}}. 
$$
Since $p>1$, choosing a sufficiently large $K$ we can make this integral as small as we wish. 
\end{proof}

Theorem \ref{th1} can be slightly extended (with essentially the same proof) to the case of  
mean $n$-valent functions in more general domains.

\begin{theorem}
\label{th4}
Let $G$ be a bounded simply connected domain and let $\phi$ be a conformal mapping from $\D$ onto $G$.
Let $f$ be a mean $n$-valent function in $G$ and assume that either

1. $f\circ \phi \in H^p$ for some $p>1$ and $\phi'\in H^\infty$,
\\
or

2. $f\circ \phi \in H^p$ with $p>2$ and $\phi'\in H^2$. 
\\
Then
$$
\int_G |f'(w)|{\rm d}\mathcal{A}(w) \le C(G,p) \sqrt{\log(n+1)} \|f\circ \phi\|_{H^p}.
$$
\end{theorem}

\begin{proof}
By the change of variable,
$$
\int_G |f'(w)|{\rm d}\mathcal{A}(w) = \int_\D |(f\circ \phi)'(z) \phi'(z)| {\rm d}\mathcal{A}(z).
$$
Now in the Case 1 the statement follows directly from Theorem \ref{th1} since $f\circ \phi$ is mean $n$-valent in $\D$.

In the Case 2 the estimate of the integral of $|(f\circ \phi)'(z) \phi'(z)|$ over the disc $\{|z| <1-1/n^K\}$ follows from Lemma \ref{lem1}.
To estimate the integral over the annulus $\{1-1/n^K \le |z| <1\}$ we apply Hayman's estimate \eqref{hain} with $\lambda=1$. 
Put $g=f\circ \phi \in H^2$. Then, choosing $\tilde r$ for a given $r$, we have
$$
\int_0^{2\pi} |g'(\tilde re^{it}) \phi'(\tilde re^{it})| {\rm d}t  \le
\bigg( \int_0^{2\pi} |g'(\tilde re^{it})|^2 |g(\tilde re^{it})|^{-1} {\rm d}t \bigg)^{1/2}
\bigg( \int_0^{2\pi} |\phi' (\tilde re^{it})|^2 |g(\tilde re^{it})|{\rm d}t \bigg)^{1/2}.
$$
Then, making use of the fact that $M(\tilde r, g) \le \|g\|_{H^p} (1-\tilde r)^{-1/p}$, we get
$$
\int_0^{2\pi} |g'(\tilde re^{it}) \phi'(\tilde re^{it})| {\rm d}t \le C(p) 
\sqrt{n} \|g\|_{H^p} \|\phi'\|_{H^2} (1-r)^{-1/2 -1/p}.
$$
Since $1/2+1/p<1$, it remains to choose a sufficiently large $K$.
\end{proof}

\begin{remark}
{\rm In the case when $f$ is a rational function, one can give a simpler proof of Theorem \ref{th1} which does not use
 the results of Hayman. For  $r\in(0,1)$ denote by $\mathcal{R}_{n,r}$ the set of rational functions of degree at most $n$ which have 
no poles in the disc $\{|z|<1/r\}$.
Sharp Bernstein-type inequalities for rational functions $f \in \mathcal{R}_{n,r}$ were found in \cite[Theorem 2.3]{BaZa}. Namely,
if $1\le p \le q \le \infty$, then
$$
\|f'\|_{H^q} \le C(p,q) \frac{n^{1+1/p-1/q}}{(1-r)^{1+1/p-1/q}} \|f\|_{H^p},
$$
while for  $1\le q \le p \le \infty$
$$
\|f'\|_{H^q} \le C(p,q) \frac{n}{(1-r)^{1+1/p-1/q}} \|f\|_{H^p}.
$$
Now, if $f\in H^p$ is a rational function, then, applying the above estimate to $f_r(z) = f(rz)$, we conclude that
$\|f_r\|_{H^1} \le \frac{n}{(1-r)^{1/p}}$, and integration with respect to $r$ over $(1-1/n^K, 1)$ for sufficiently
large $K$ gives the required estimate for the boundary annulus. 

Note that while the estimates in \cite[Theorem 2.3]{BaZa} are best possible, the function $f_r$ 
has more regularity than a general function in $\mathcal{R}_{n,r}$ since it is in $H^p$ in a larger disc. 

It is an interesting open problem to find the sharp growth of means in Theorem \ref{th1} for rational functions in $H^1$.}
\end{remark}

\begin{remark}
{\rm In \cite{bz19} the following inequality was proved for rational functions of degree at most $n$:
$\|f'\|_{A^1(\D)}\le C \log n \, \|f\|_{BMOA}$, where $BMOA$ denotes the
space of analytic functions of bounded mean oscillation in $\D$, and the question about sharpness of this
inequality was posed. Since $BMOA \subset H^p(\D)$ for any $p<\infty$, we know now that a correct and sharp
inequality is $\|f'\|_{A^1(\D)} \le C \sqrt{\log n} \|f\|_{BMOA}$ and it holds for
any mean $n$-valent function.

It should be mentioned that it is much easier to show the sharpness of this inequality in $BMOA$ (and thus in all $H^p$
with $p<\infty$), than for $p=\infty$, where the proof was based on a deep and implicit construction
of Bloch space functions from \cite{BanMoo}. Indeed, consider the polynomials
$$
P_n(z) = \sum_{k=1}^n z^{2^k}.
$$
It is easy to see (see, e.g., Lemma \ref{lac} below) that $\|f'\|_{A^1(\D)} \asymp n$.
On the other hand, it is well known (see \cite[Theorem 9.3]{gir}) that for a lacunary series, 
$\|f\|_{BMOA} \asymp \|f\|_{H^2(\D)}$, and so in our case $\|f\|_{BMOA} \asymp \sqrt{n}$.
Here $a \asymp b$ means that the ratio $a/b$ is bounded away from zero and infinity by some absolute constants.
Thus, $\|f'\|_{A^1(\D)}/\|f\|_{BMOA} \ge c \sqrt{n} = c\sqrt{\log {\rm deg} P_n}$ for some absolute $c>0$.}
\end{remark}
\bigskip


\section{Proof of Theorem \ref{th2}}

The estimate over the disc $\{|z| < 1-1/n\}$ follows from Lemma \ref{lem1-0} with $K=1$. To estimate the integral over the
annulus $\{1-1/n \le |z| <1 \}$ we use again Hayman's estimate \eqref{hain}. By the H\"older inequality 
with exponents $2/\sigma$ and $2/(2-\sigma)$,
$$
\int_0^{2\pi}|f'(re^{it})|^\sigma {\rm d}t \le \bigg( \int_0^{2\pi}|f'(re^{it})|^2 
|f(re^{it})|^{\lambda - 2} {\rm d}t \bigg)^{\sigma/2}
\bigg( \int_0^{2\pi}|f(re^{it})|^{(2- \lambda)\sigma/(2-\sigma)} {\rm d}t \bigg)^{1-\sigma/2}
$$
for any $\lambda\in (0,2)$. Assume that $p\ge (2- \lambda)\frac{\sigma}{2-\sigma}$, 
then the second factor above
does not exceed $\|f\|_{H^p}^{(2-\lambda)\sigma/2}$. Making use of the fact that 
 $M(r,f) \le C(p) \frac{\|f\|_{H^p}}{(1-r)^{1/p}}$ and arguing as in the proof of Theorem \ref{th1},
we conclude that
$$
\int_0^{2\pi}|f'(re^{it})|^\sigma {\rm d}t \le C(p)\|f\|_{H^p}^{\sigma} \frac{n^{\sigma/2}}{\lambda^{\sigma/2}
(1-r)^{(1+\lambda/p)\sigma/2}},\qquad r\in(0,1/2).
$$
Hence, 
$$
\|f\|^\sigma_{B_\sigma^\alpha} \le C(p) \|f\|_{H^p}^{\sigma} \frac{n^{\sigma/2}}{\lambda^{\sigma/2}} 
\int_{1-1/n}^1 \frac{{\rm d}r}{(1-r)^\beta},
$$
where
$$
\beta = 1 + (1+\lambda/p)\sigma/2 - (1-\alpha)\sigma.
$$
The integral will converge whenever $p(1-2\alpha) >\lambda$. Thus, we have the estimate
\begin{equation}
\label{hosk}
\|f\|^\sigma_{B_\sigma^\alpha} \le C(p, \alpha,\sigma) \frac{n^{\alpha\sigma + \lambda\sigma/(2p)}}{\lambda^{\sigma/2}} \|f\|_{H^p}^{\sigma},
\end{equation}
when the conditions
\begin{equation}
\label{hos}
p\ge (2- \lambda)\frac{\sigma}{2-\sigma},  \qquad p(1-2\alpha) >\lambda
\end{equation}
are satisfied. 

Assume first that $p\ge \frac{2\sigma}{2-\sigma}$. In this case inequalities \eqref{hos}
will be satisfied for $\lambda = \frac{1}{\log n}$
and so $n^{\lambda\sigma/(2p)} = O(1)$, $n\to \infty$. This gives \eqref{gott1}.

Now assume that $p< \frac{2\sigma}{2-\sigma}$ and define $\lambda$ by the equation
$p = (2- \lambda)\frac{\sigma}{2-\sigma}$. Since in this case we assume that $\sigma < \frac{p}{\alpha p+1}$,
it is easy to see that $p(1-2\alpha) >\lambda$. Inserting this $\lambda$ into \eqref{hosk}
we obtain inequality \eqref{gott2}.
\qed
\bigskip


\section{Proof of Theorem \ref{th3}}

Fix $K>0$. Note that $B'g=(Bg)'-Bg'$ and so the estimate of the integral 
$\int_{|z|<1-1/n^K} |B'(z) g(z)|{\rm d}\mathcal{A}(z)$
follows from Lemma \ref{lem1}.

Let us estimate the integral over the annulus $\{1-1/n^{K}\le|z|<1\}$.
We have 
\[
\int_{0}^{1}\int_{0}^{2\pi}|B'(re^{it})g(re^{it})|r {\rm d}t{\rm d}r \leq\|g\|_{H^{p}}
\int_{0}^{1}\bigg(\int_{0}^{2\pi}|B'(re^{it})|^{p/(p-1)} {\rm d}t \bigg)^{1-1/p}
{\rm d}r.
\]
It follows from the inequalities 
\[
|B'(re^{it})|\leq \frac{1}{1-r^2}, \qquad 
\int_{0}^{2\pi}|B'(re^{it})|{\rm d}t\leq  2\pi n,
\]
that for any $q\in[1,\infty)$,
$$
\int_{0}^{2\pi}|B'(re^{it})|^{q}{\rm d}t  \leq\frac{1}{(1-r^{2})^{q-1}}\int_{0}^{2\pi}|B'(re^{it})|{\rm d}t
 \leq\frac{2\pi n}{(1-r^{2})^{q-1}},
$$
and, in particular, choosing $q=p/(p-1)$ we get
\[
\bigg(\int_{0}^{2\pi} |B'(re^{it})|^{p/(p-1)} {\rm d}t \bigg)^{1-1/p} 
\leq \bigg(\frac{2\pi n}{(1-r^2)^{1/(p-1)}} \bigg)^{1-1/p} \leq \frac{(2\pi n)^{1-1/p}}{(1-r)^{1/p}}.
\]
Since $1/p<1$, the integral over $r\in(1-1/n^{K},1)$ can be made
arbitrarily small. \qed
\bigskip


\section{Application to an inverse theorem of rational approximation }

In this section we prove Theorem \ref{tinv}. The proof follows 
a standard ``dyadic'' scheme used, e.g., in \cite{Dol}, however with a double exponentiation.

In what follows for two positive functions $a$ and $b$, we say that $a$ is dominated
by $b$, denoted by $a\lesssim b$, if there is a constant $C>0$
such that $a\leq Cb$; we say that $a$ and $b$ are comparable, denoted
by $a\asymp b$, if both $a\lesssim b$ and $b\lesssim a$. 

For any $n$ take $f_{n}\in\mathcal{R}_{n}$ such that $\|f-f_{n}\|_{H^{p}}\leq 2R_{n}(f,\,H^{p})$
and write 
$$
f_{2^{2^k}} = f_2+ \sum_{m=1}^k u_m, \qquad u_m = f_{2^{2^m}} - f_{2^{2^{m-1}}}.
$$
Since $f_{2^{2^k}}$ converge to $f$ in $H^p$ and uniformly on compact
subsets of $\mathbb{D}$, we have
$$
\|f'\|_{A^1(\D)} \le \liminf_{k\to \infty}\big\|f_{2^{2^k}}'\big\|_{A^1(\D)} \le \|f_2'\|_{A^1(\D)} +
\liminf_{k\to \infty} \sum_{m=1}^k \|u_m'\|_{A^1(\D)}.
$$ 
Clearly, we have $\|u_m\|_{H^p} = \big\|f_{2^{2^m}} - f_{2^{2^{m-1}}}\big\|_{H^p} \le 4 
R_{2^{2^{m-1}}}(f,\,H^{p})$.

Since $u_m \in \mathcal{R}_{2^{2^m}}(\D)$, by Theorem \ref{th1} we have
$$
\|u_m'\|_{A^1(\D)} \lesssim \sqrt{\log 2^{2^m}} \|u_m\|_{H^p} \lesssim 2^{m/2} R_{2^{2^{m-1}}}(f,\,H^{p}).
$$
Thus, $f'\in A^1(\D)$ whenever
$$
\sum_{m=1}^\infty 2^{m/2} R_{2^{2^m}}(f,\,H^{p}) <\infty.
$$
It remains to show that the latter condition is implied by (actually is equivalent to) the condition
$$
\sum_{n=1}^\infty \frac{R_n(f,\,H^{p})}{n\sqrt{\log n}}<\infty.
$$
This is obvious, since 
$$
\sum_{n=2^{2^{m-1}}}^{2^{2^{m}}}  \frac{1}{n\sqrt{\log n}} \asymp
\int_{2^{2^{m-1}}}^{2^{2^{m}}}  \frac{{\rm d}t}{t\sqrt{\log t}} \asymp 2^{m/2}.
$$
\qed

For general domains we may repeat the above argument to obtain the following result generalizing
the inverse approximation theorems due to Dolzhenko \cite{Dol} (see \eqref{tig}). 
Here we need to assume that $f\in H^\infty(G)$.

\begin{theorem}
\label{jkl}
Let $G$ be a bounded domain in $\mathbb{C}$  with a rectifiable boundary, $f\in H^\infty(G)$.
If 
$$
\sum_{n=1}^\infty \frac{R_n(f, H^\infty(G))}{n\sqrt{\log n}} <\infty,
$$ 
then $f'\in A^1(G)$. If $1 < p\le 2$  and  
$$
\sum_{n=1}^\infty \frac{R_n(f, H^\infty(G))}{n^{1/p}} <\infty,
$$ 
then $f'\in A^p(G)$. Moreover, $R_n(f, H^\infty(G))$ can be replaced by a smaller quantity
$V_n(f, H^\infty(G))$.
\end{theorem}

In the case $1<p\le 2$ one should consider approximants $f_{2^m}$ in place of $f_{2^{2^m}}$. Here 
the novelty is in a larger class of domains without any regularity except the rectifiable boundary.
\medskip

It remains to show the sharpness of the condition \eqref{et}. We will use the following simple
(and apparently well-known) lemma.

\begin{lemma}
\label{lac}
If  $f(z) = \sum_{k\ge 1} a_k z^{2^k} \in A^1(\D)$, then $\sum_{k\ge 1} 2^{-k} |a_k| <\infty$.
In particular, if $f'\in A^1(\D)$, then $\sum_{k\ge 1} |a_k| <\infty$.
\end{lemma}

\begin{proof}
Obviously, for $1-2^{-k} \le r \le 1-2^{-k-1}$, we have
$$
|a_k|r^{2^k} =\bigg|\int_0^{2\pi} f(re^{i\phi}) e^{-i 2^k \phi} {\rm d}\phi\bigg| \le 
\int_0^{2\pi} |f(re^{i\phi})| {\rm d}\phi,
$$
whence, by integration,  $2^{-k} |a_k| \lesssim \int_{1-2^{-k} \le |z| \le 1-2^{-k-1}} |f(z)|{\rm d}\mathcal{A}(z)$.
\end{proof}

\begin{proof}[Proof of Proposition \ref{tinw}]
Without loss of generality we may assume that $\phi$ is nonincreasing (otherwise replace $\phi$
by $\tilde \phi(x) = \sup_{t\ge x} \phi(t)$). Then we can construct a nonincreasing positive function $\psi$
such that 
$$
\sum_{k\ge 1} \frac{\psi(k)}{k}=\infty, \qquad
\sum_{k\ge 1} \frac{\psi(k)\phi(2^k)}{k}<\infty,
$$
and, following the construction of Littlewood \cite{Lit}, put
$$
f(z)=\sum_{j=1}^\infty \frac{\psi(j)}{j} z^{2^j}.
$$
Obviously, $f\in H^2$, and, moreover, $f\in H^p$, $1\le p<\infty$,
since, by the classical Paley--Kahane--Khintchine inequality,
$\|g\|_{H^p} \asymp \|g\|_{H^2}$ for any lacunary series $g$
(see \cite[Theorem 2.23]{Pav} or \cite[Ch. 5, Th. 8.20]{Zyg}).
Also, $f'\notin A^1(\D)$ by Lemma \ref{lac}. 

Now let us estimate the best polynomial approximation $E_n(f, H^p)$. Obviously,
for $2^k \le n <2^{k+1} -1$ we have
$$
E_n^2(f, H^2)  \le \bigg\| \sum_{j=k+1}^\infty \frac{\psi(j)}{j} z^{2^j} \bigg\|^2_{H^2} 
= \sum_{j=k+1}^\infty \frac{(\psi(j))^2}{j^2} \lesssim \frac{(\psi(k))^2}{k}. 
$$
By the Paley--Kahane--Khintchine inequality, for $1<p<\infty$,
$$
E_n(f, H^p) \lesssim \bigg\| \sum_{j=k+1}^\infty \frac{\psi(j)}{j} z^{2^j}  \bigg\|_{H^p} 
\lesssim \frac{\psi(k)}{\sqrt{k}},
$$
and so 
$$
\begin{aligned}
\sum_{n=2}^\infty \frac{E_n(f, H^p)}{n\sqrt{\log n}} \phi(n) & = 
\sum_{k=1}^\infty  \sum_{n=2^k}^{2^{k+1} -1} \frac{E_n(f, H^p)}{n\sqrt{\log n}} \,\phi(n) \\
& \lesssim
\sum_{k=1}^\infty 2^k\cdot \frac{\psi(k)}{\sqrt{k}}\cdot \frac{1}{2^k \sqrt{k}} \,\phi(2^k) =
\sum_{k=1}^\infty \frac{\psi(k)  \phi(2^k)}{k}  <\infty.
\end{aligned}
$$
\end{proof}

\end{document}